\documentclass[reqno,11pt]{amsart}
\usepackage{amsmath,amsfonts,amssymb,amsthm}

\usepackage{latexsym,bm,graphicx}
\usepackage{mathrsfs}
\usepackage{color}
\title{Sharp Spectral Gaps on Metric Measure Spaces}
\author{Yin Jiang}
\address{Department of Mathematics, Sun Yat-sen University, Guangzhou, 510275,China}
\email[Yin Jiang]{Jiangy39@mail2.sysu.edu.cn}
\author{Hui-Chun Zhang}
\address{Department of Mathematics, Sun Yat-sen University, Guangzhou, 510275,China}
\email[Hui-Chun Zhang]{zhanghc3@mail.sysu.edu.cn}

\newtheorem{thm}{Theorem}[section]
\newtheorem{prop}[thm]{Proposition}
\newtheorem{lem}[thm]{Lemma}
\newtheorem{cor}[thm]{Corollary}

\theoremstyle{definition}
\theoremstyle{remark}

\newtheorem{defn}[thm]{Definition}
\newtheorem{remark}[thm]{Remark}

\numberwithin{equation}{section}

\newcommand{\ls}{\leqslant}
\newcommand{\gs}{\geqslant}

\newcommand{\rd}{\mathrm{d}}
\newcommand{\cE}{\mathcal{E}}
\newcommand{\bV}{\mathbb{V}}
\newcommand{\bVi}{\mathbb{V}_{\kern-2pt\infty}}

\newcommand{\bR}{\mathbb{R}}
\newcommand{\fs}{\mathfrak{s}}

\newcommand{\eps}{\epsilon}

\newcommand{\hla}{\hat{\lambda}}

\def\Xint#1{\mathchoice
{\XXint\displaystyle\textstyle{#1}}%
{\XXint\textstyle\scriptstyle{#1}}%
{\XXint\scriptstyle\scriptscriptstyle{#1}}%
{\XXint\scriptscriptstyle\scriptscriptstyle{#1}}%
\!\int}
\def\XXint#1#2#3{{\setbox0=\hbox{$#1{#2#3}{\int}$ }
\vcenter{\hbox{$#2#3$ }}\kern-.6\wd0}}

\def\dashint{\Xint-}

\begin{document}

\begin{abstract}
In this paper, we extend the sharp lower bounds of spectal gap, due to Chen-Wang \cite{chen1993application,chen1997general}, Bakry-Qian \cite{bakry2000some} and Andrews-Clutterbuck \cite{andrews2013sharp}, from smooth Riemaniannian manifolds to general metric measure spaces with Riemannian curvature-dimension condition $RCD^*(K,N)$.

\end{abstract}
\maketitle

\section{introduction}
Let $(X,d,m)$ be a compact metric measure space. Given a Lipschitz function $f:X\to \mathbb R$, its point-wise Lipschitz constant ${\rm Lip}f(x)$ is defined as
$${\rm Lip}f(x):=\limsup_{y\to x}\frac{|f(y)-f(x)|}{d(x,y)}.$$
In this paper, we are concerned with the spectral gap
\begin{equation}\label{eq1.1}
\lambda_1(X):=\inf\Big\{\frac{\int_X ({\rm Lip}f)^2\rd m}{\int_Xf^2\rd m}:\ f\in Lip(X)\backslash\{0\} \quad {\rm and}\quad \int_Xf\rd m=0\Big\},
\end{equation}
where $Lip(X)$ is the set of Lipschitz functions on $X$.

When $M$ is a compact smooth Riemannian manifold without boundary (or with a convex boundary $\partial M$), the study of the lower bounds of the first eigenvalue $\lambda_1$ of the Laplace-Beltrami operator $\Delta$ has a long history. See for example, Lichnerowicz \cite{lichnerowicz1958geometrie}, Cheeger \cite{cheeger1970lower}, Li-Yau \cite{li1980estimates}, and so on. For an overview the reader is referred to the introduction of \cite{besson2004isoperimetric,bakry2000some,ledoux2004spectral} and Chapter 3 in book \cite{schoen1994lectures}, and references therein.
In particular the following comparison theorem for $\lambda_1$ has been established by Chen-Wang \cite{chen1993application,chen1997general}, Bakry-Qian \cite{bakry2000some} and Andrews-Clutterbuck \cite{andrews2013sharp} independently, via three different methods.

\begin{thm} [Chen-Wang \cite{chen1993application,chen1997general}, Bakry-Qian \cite{bakry2000some}, Andrews-Clutterbuck \cite{andrews2013sharp}]\label{in:BQ}
Let $M$ be an $n$-dimensional compact Riemannian manifold without boundary (or with a convex boundary). Suppose that the Ricci curvature $Ric(M)\gs K$ and that
 the diameter $\ls d$. Let $\lambda_1$ be the first (non-zero) eigenvalue (with Neumann boundary condition if the boundary is not empty). Then
\[
\lambda_1(M) \gs \hla(K,N,d)
\]
where $\hla(K,N,d)$ denotes the first non-zero Neumann eigenvalue of the following one-dimensional model:
\[
v''(x)-(N-1)T(x)v'(x)=-\lambda v(x) \qquad x\in (-\frac{d}{2}, \frac{d}{2}), \qquad v'(-\frac{d}{2})=v'(\frac{d}{2})=0
\]
and
\[
T(x)=\left\{
\begin{array}{ll}
\sqrt{\frac{K}{N-1}} \tan(\sqrt{\frac{K}{N-1}} x)& \qquad \text{ \rm if }K \gs 0,\\
\sqrt{\frac{-K}{N-1}} \tan(\sqrt{\frac{-K}{N-1}} x)& \qquad \text{ \rm if }K < 0.
\end{array}
\right.
\]
\end{thm}
 This comparison Theorem \ref{in:BQ} implies the classical Lichnerowicz's estimate \cite{lichnerowicz1958geometrie} for $K=n-1$ and also Zhong-Yang's estimate \cite{ZhongYang1984estimate} for $K=0$.
Some lower bounds of the spectral gaps have been extended to singular spaces. In \cite{shio99}, Shioya discussed spectral gaps in Rimannian orbifolds. In \cite{petruninharmonic}, Petrunin proved the Linchnerowiz's estimate to Alexandrov spaces with curvature $\gs 1$ in the sense of Alrexandrov. Recently, Theorem \ref{in:BQ} has been extended to Alexandrov spaces in \cite{qian2013sharp} using a notion of generalized lower Ricci curvature bounds in \cite{huichun2010new}, and Wang-Xia \cite{wang2013sharp} to Finsler manifolds.

In the last few years, several notions for ``the generalized Ricci curvature bounded below" on general metric spaces have been introduced.
Sturm \cite{sturm2006geometry,sturm2006geometry2} and Lott-Villani \cite{lott2009ricci}, independently, introduced a so-called \emph{curvature-dimension condition}, denoted by $CD$, on metric measure spaces via optimal transprotation. A refinement for this notion is given in Ambrosio-Gigli-Savar\'e \cite{ambrosio2014metric}, which is called \emph{Remannian curvature-dimension condition}, denoted by $RCD^*$. Recently, in two remarkable works, Ambrosio-Gigli-Savar\'e \cite{ambrosio2015bakry} and Erbar-Kuwada-Sturm \cite{erbar2013equivalence}, they proved the equivalence of the Remannian curvature-dimension condition and of the Bochner formular of Bakry-\'Emery via an abstract $\Gamma_2$-calculus, denoted by $BE$. Notice that in the case where $M$ is a (compact) Riemannian manifold.  Given two numbers $K\in \mathbb R$ and $N\gs1$, $M$ satisfying the Remannian curvature-dimension condition $RCD^*(K,N)$ is equivalent to that the Ricci curvature $Ric(M)\gs K$ and the dimension $dim\ls N$.

We will consider the spectral gap on metric measure spaces under a suitable Remannian curvature-dimension condition. Lott-Villani \cite{lott2007weak} and
Erbar-Kuwada-Sturm \cite{erbar2013equivalence} extended Linchnerowicz's estimate to metric measure spaces with $CD(K,N)$ or $RCD^*(K,N)$ for  $K>0$ and  $1\ls N <\infty$.  In this paper, we will extend Theorem \ref{in:BQ} to general metric measure spaces. Precisely, we have the following theorem.
\begin{thm} \label{thm:main}
Let $K\in \mathbb R$, $1\ls N < \infty$ and $d>0$. Let $(X,d,m)$ be a compact metric measure space satisfying the Remannian curvature-dimension condition $RCD^*(K,N)$ and the diameter $\ls d.$  Then the spectral gap $\lambda_1(X)$ has the following lower bound
\begin{equation}\label{eq1.2}
\lambda_1(X) \gs \hla(K,N,d),
\end{equation}
where $\hla(K,N,d)$ is given in Theorem \ref{in:BQ}.
\end{thm}
Our proof of Theorem \ref{thm:main} relies on the self-improvement of regularity under the Riemannian curvature-dimension condition  (Theorem \ref{thm:sel}) and a version of maximum principle, which is similar as the classical maximum principle for $C^2$-functions on mainifolds (see Proposition \ref{prop3.1} and Remark \ref{rem3.2}).

\begin{remark}When $N>1$ and $K=N-1$, the above Theorem \ref{thm:main} implies that
 $$\lambda_1(X)\gs \frac{N}{1-\cos^N(d/2)}.$$
In particular, this gives that if $\lambda_1(X)=N$, then $d=\pi$.  The combination of this and the maximal diameter theorem in \cite{ketterer2014cones} implies a Obata-type rigidity theorem for general metric measure spaces, which is also proved in \cite{ketterer2014obata} by Ketterer, independently.
\end{remark}

\section{Preliminaries}

In this section, we recall some basic notions and the calculus on metric measure spaces. For our purpose in this paper, we will focus only on the case of compact spaces.
Let $(X,d)$ be a compact metric space, and let $m$ be a Radon measure with ${\rm supp}(m)=X$.

\subsection{Riemanian curvature-dimension condition $RCD^*(K,N)$} $\ $

Let $(X,d,m)$ be a compact metric measure space. The Cheeger energy is given in \cite{ambrosio2014calculus} from the relaxation in $L^2(X,m)$ of the point-wise Lipschitz constant of Lipschitz functions. That is,
given a function $f\in L^2(X,m)$, the Cheeger energy of $f$ is defined \cite{ambrosio2014calculus} by
$${\rm Ch}(f):=\inf\Big\{\liminf_{j\to\infty}\frac{1}{2}\int_X({\rm Lip} f_j)^2dm\Big\},$$
where the infimum is taken over all sequences of Lipschitz functions $\{f_j\}$ converging to $f$ in $L^2(X,m)$. If ${\rm Ch}(f)<\infty$, then there is a (unique) so-called minimal
relaxed gradient $|Df|_w$ such that
$${\rm Ch}(f)=\frac{1}{2}\int_X|Df|_w^2\rd m.$$
The domain of ${\rm Ch}$ in $L^2(X,m)$,  ${\rm D(Ch)}$, is a Banach space with norm $\|f\|_{L^2}^2+\| |Df|_w\|_{L^2}^2$.
\begin{defn} (\cite{ambrosio2014calculus})\indent
 A metric measure space $(X,d,m)$ is called \emph{infinitesimally Hilbertian}  if the associated Cheeger energy ${\rm Ch}$ is a quadratic form.
\end{defn}
Let $(X,d,m)$ be an infinitesimally Hilbertian space. It is proved in \cite{ambrosio2014metric} that the scalar product
$$\Gamma(f,g):=\lim_{\epsilon\to0^+}\frac{|D(f+\epsilon g)|^2_w-|Df|^2_w}{2\epsilon}\qquad f,g\in {\rm D(Ch)}$$
exists in $L^1(X,m).$
In the following we denote by the Hilbert space $\mathbb V:={\rm D(Ch)}$ with the scalar product
$$(f,g)_{\mathbb V}:=\int_X\big(fg+\Gamma(f,g)\big)\rd m.$$

The quadratic form ${\rm Ch}$ canonically induces a symmetric, regular, strongly local Dirichlet form $({\rm Ch},\mathbb V)$. The regular property of $({\rm Ch},\mathbb V)$ comes from that $X$ is always assumed to be compact.
Moreover, for any $f,g\in \mathbb V$,  $\Gamma(f,g)$ provides an explicit expression of the \emph{Carr\'edu Champ} of the Dirichlet form $({\rm Ch},\mathbb V)$. The associated energy measure of $f$ is absolutely continuous with respect to $m$ with density
$\Gamma(f)=|Df|^2_w$.

Denote by $(H_t)_{t>0}$ and $\Delta$ the associated Markov semigroup in $L^2(X,m)$ and its generator respectively. Since $X$ is compact, according to \cite{rajala2012local}, the $RCD^*(K,N)$ condition implies that $(X,d,m)$ supports a global Poincar\'e inequality. Moreover, the operator $(-\Delta)^{-1}$ is a compact operator. Then the spectral theorem gives that the $\lambda_1(X)$ in \eqref{eq1.1} is the first non-zero eigenvalue of $-\Delta.$ (See, for example, \cite{davies}.)

 We adopt the notations given in \cite{ambrosio2013bakry}:
$$D_{\mathbb V}(\Delta):=
\big\{f\in \mathbb V  :\Delta f \in  \mathbb V \big\}$$
and, for every $p\in[1,\infty]$,
$$D_{L^p}(\Delta):=
 \big\{f\in \mathbb V\cap L^p(X,m) :\Delta f \in  L^2\cap L^p(X,m) \big\}. $$

\begin{defn}(\cite{ambrosio2013bakry,erbar2013equivalence})\indent
Let $K\in \mathbb R$ and $N\gs 1$. An infinitesimally Hilbertian space $(X,d,m)$ is said to satisfy $BE(K,N)$ \emph{condition} if the associated Dirichlet form $({\rm Ch},\mathbb V)$ satisfies
$$\int_X\Big(\frac{1}{2} \Gamma(f)\Delta\phi-\Gamma(f,\Delta f)\phi\Big)\rd m\gs K\int_X\Gamma (f)\phi \rd m +\frac{1}{N}\int_X(\Delta f)^2\phi \rd m$$
for all $f\in D_{\mathbb V}(\Delta)$ and all nonnegative $\phi\in D_{L^\infty}(\Delta).$
\end{defn}

According to \cite{ambrosio2013bakry,erbar2013equivalence}, the Riemanian curvature-dimension condition $RCD^*(K,N)$ is equivalent to the corresponding Bakry-\'Emery condition $BE(K,N)$ with a slight regularity. We shall use the following definition for $RCD^*(K,N)$ (Notice that $X$ is always assumed to be compact in the paper).
\begin{defn} (\cite{ambrosio2013bakry,erbar2013equivalence}),\indent
Let $K\in \mathbb R$ and $N\gs1$. A compact, infinitesimally Hilbertian length space  $(X,d,m)$
is said to satisfy $RCD^*(K,N)$ \emph{condition} (or metric $BE(K,N)$ condition) if it satisfies $BE(K,N)$ and that every $f\in \mathbb V$ with $\|\Gamma(f)\|_{L^\infty}\ls1$ has a 1-Lipschitz representative.
\end{defn}
Recall that a (locally) compact metric $(X,d)$ is a length space if the distance between any two points in $X$ can be realized as the length of some curve connecting them.
Notice that if $(X,d,m)$ satisfies $RCD^*(K,N)$  condition then $d=d_{\rm Ch}$, where $d_{\rm Ch}$ is the induced metric by the Dirichlet form $({\rm Ch},\mathbb V)$.
For any $f\in \mathbb V$ with $\Gamma(f)\in L^\infty(X,m)$, we always identify $f$ with its Lipschitz representative. Moreover, $H_tf$, $H_t(|\nabla f|^2_w)$ and $\Delta H_tf$ have continuous representatives (see Proposition 4.4 of \cite{erbar2013equivalence}).

\subsection{The self-improvement of regularity on $RCD^*(K,N)$-spaces} $\ $

Let $K\in \mathbb R$ and $1\ls N<\infty$, and let $(X,d,m)$ be a compact metric measure space satisfying $RCD^*(K,N)$ condition.

Let us recall an extension of the generator $\Delta$ of $({\rm Ch},\mathbb V)$, which is introduced in \cite{ambrosio2013bakry,savare2013self}.
Denote by $\mathbb V'$ the set of continuous linear functionals $\ell :\bV\to\mathbb R,$ and $\bV'_+$ denotes the set of positive linear fucntionals $\ell \in \bV'$ such that $\ell( \varphi) \gs 0$ for all $\varphi \in \bV$ with  $ \varphi \gs 0$   m-a.e.  in $X $. An important characterization of functionals in $\mathbb V'_+$ is that, for each $\ell\in \mathbb V'_+$ there exists a unique corresponding Radon measure $\mu_\ell$ on $X$ such that
$$\ell(\varphi)=\int_X\tilde\varphi\rd\mu_\ell\quad \forall \varphi\in\mathbb V,$$
where $\tilde{\varphi}$ is a quasi continuous representative of $\varphi$. Denote by
$$\mathbb{M}_{\infty}:=\Big\{f\in \bV\cap L^{\infty}(X,m): \exists\ \mu
\ \ {\rm such\ that}\ \ -\cE(f,\varphi)=\int_X \tilde{\varphi} \rd \mu \quad \forall \varphi \in \bV\Big\},$$
where $\mu=\mu_+-\mu_-$ with $\mu_{+},\mu_- \in \bV'_+$.
When a function $f\in \mathbb{M}_{\infty}$, we set $\Delta^*f:=\mu,$ and denote its Lebesgue's decomposition w.r.t $m$ as $\Delta^* f=\Delta^{ab} f\cdot m+\Delta^sf$.
It is clear that if $f\in D(\Delta)\cap L^{\infty}(X,m)$ then $f\in \mathbb{M}_{\infty}$ and $\Delta^*f=\Delta f\cdot m.$

\begin{lem}\label{lem2.1} Let $K\in\mathbb R$ and $N\gs1$,  and let $(X,d,m)$ be a compact metric measure space  satisfying $RCD^*(K,N)$ condition.\\
{\rm(i)\ (Chain rule, \cite[Lemma 3.2]{savare2013self})}\ \ \  If $g\in  D(\Delta)\cap Lip(X)$ and $\phi\in C^{2}(\mathbb R)$ with $\phi(0)=0$, then we have
\begin{equation*}
\phi\circ g\in D(\Delta)\cap Lip(X)\qquad {\rm and}\qquad\Delta (\phi\circ g)=\phi'\circ g\cdot\Delta g +\phi''\circ g\cdot\Gamma(g);
\end{equation*}
{\rm(ii)\ (Leibniz rule, \cite[Corollary 2.7]{savare2013self})}\ \ \ If $g_1\in \mathbb{M}_{\infty}$ and $g_2\in D(\Delta)\cap Lip(X)$, then we have
\begin{equation*}
g_1\cdot g_2\in \mathbb{M}_{\infty}\qquad {\rm and}\qquad \Delta^*(g_1\cdot g_2)=g_2\cdot\Delta^*g_1+g_1\cdot\Delta g_2\cdot m+2\Gamma(g_1,g_2)\cdot m.
\end{equation*}
\end{lem}
\begin{remark}\label{rem2.5}
We can take $\phi\in C^{2}(\mathbb R)$ without the restriction $\phi(0)=0$ in the Chain rule. This comes from
 the fact that $1\in D(\Delta)$ and $\Delta 1=0$, because $X$ is assumed to be compact.
 \end{remark}

The following self-improvement of regularity is given in Lemma 3.2 of \cite{savare2013self}. (See also Theorem 2.7 of \cite{garofalo2014li}).

\begin{thm} $($\cite{savare2013self,garofalo2014li}$)$ \label{thm:sel} \indent Let $K\in \mathbb R$ and $1\ls N<\infty$, and let $(X,d,m)$ be a compact metric measure space satisfying $RCD^*(K,N)$ condition.
If $f \in D_{\bV}(\Delta)\cap Lip(X)$, then we have $\Gamma(f) \in \mathbb{M}_{\infty}$ and
\begin{equation} \label{in:self}
\frac12 \Delta^*\Gamma(f)-\Gamma(f,\Delta f)\cdot m \gs K\Gamma(f)\cdot m +\frac{1}{N} (\Delta f)^2\cdot m.
\end{equation}
\end{thm}

A crucial fact, which is implied by the above inequality, is that the singular part of $\Delta^*f$ has a correct sign: $\Delta^s \Gamma(f)$ is non-negative.

Using the same trick as in the proof of Bakry-Qian \cite[Thm 6]{bakry2000some} and \cite[Thm 3.4]{savare2013self}, one can prove the following Corollary of Theorem \ref{thm:sel}:
\begin{cor}\label{ineq:sfi2}
Let $K\in \mathbb R$ and $1\ls N<\infty$, and let $(X,d,m)$ be a compact metric measure space satisfying $RCD^*(K,N)$ condition.
If $f \in D_{\bV}(\Delta)\cap Lip(X)$, then $\Delta^s f\gs 0$ and the following holds $m$-a.e on $\{x\in X:\Gamma(f)(x)\neq 0\}$,
\begin{equation}\label{in:modi}
\left(\frac12\Delta^{ab}\Gamma(f)-\Gamma(f,\Delta f)-K\Gamma(f)-\frac1N(\Delta f)^2\right)\gs
\frac{N}{N-1}\left(\frac{\Delta f}{N}  -\frac{\Gamma(f,\Gamma(f))}{2\Gamma(f)}\right)^2.
\end{equation}
\end{cor}

For $\kappa \in \bR$ and $\theta \gs 0$ we denote the function
\[
\fs_\kappa (\theta)=\left\{
\begin{array}{ll}
\frac{1}{\sqrt \kappa} \sin (\sqrt \kappa \theta),&\qquad \kappa>0,\\
\theta,&\qquad \kappa=0,\\
\frac{1}{\sqrt{ -\kappa}} \sinh (\sqrt{-\kappa} \theta),&\qquad \kappa<0.
\end{array}
\right.
\]

\begin{prop}[Bishop-Gromov inequality, \cite{gigli2013optimal,sturm2006geometry2}]
For each $x_0 \in X$ and $0<r<R \ls \pi \sqrt{(N-1)/(K\vee 0)}$, we have
\begin{equation}\label{ineq:BG}
\frac{m(B_r(x_0))}{m(B_R(x_0))} \gs \frac{\int_0^r \fs_{\frac{K}{N-1}}(t)^{N-1} \rd t}{\int_0^R \fs_{\frac{K}{N-1}}(t)^{N-1} \rd t}.
\end{equation}
\end{prop}
\begin{proof}
By Corollary of 1.5 in \cite{gigli2013optimal}, $(X,d,m)$ satisfies $MCP(K,N)$ condition.  the desired Bishop-Gromov inequality \eqref{ineq:BG} holds on $MCP(K,N)$-spaces by Remark 5.3 of \cite{sturm2006geometry2}.
\end{proof}

%\begin{prop} [heat kernel lower bound]
%Let $h(t,x,y)$ be the heat kernel of $H_t$. By Theorem 4.8 of \cite{sturm1996analysis}, there exists a constant $C=C(K,N,\text{diam}X)$ such that
%\[
%h(t,x,y)>\frac{1}{C}\frac{1}{m (B_{\sqrt{t}}(x))} \exp\left(-C \frac{d^2(x,y)}{t}\right)
%\]
%for all $t>0$ and all $x,y\in X$.
%\end{prop}

We need also  the following  mean value inequality in \cite{mondino2014structure}. See also Lemma 2.1 of \cite{colding2011sharp}.
\begin{lem}\cite[Lemma 3.4]{mondino2014structure}
Let $f \in D(\Delta)$ be a non-negative, continuous function with $\Delta f \ls c_0$ $m$-a.e. Then there exists a constant $C(K,N,\text{diam}X)$ such that the following holds:
\begin{equation}\label{in:mv}
\dashint_{B_r(x)} f \rd m
\ls C (f(x)+c_0 r^2).
\end{equation}
\end{lem}
At last, we need the following Sobolev inequality, whose proof is similar to that of Theorem 13.1 of \cite{hajlasz2000sobolev}. For the reader's convenience, we include a proof here.
\begin{lem}\label{lem2.9}
Let $E \subset X$ be an $m$-measurable subset with $m(E)>0$. Then there exists constants $\nu > 2$ and $\widetilde{C_S}$ which depends only on $K$, $N$, $X$ and $E$, such that for any $f \in \bV$ with $f=0$ $m$-a.e. in $E$, the following Sobolev inequality holds:
\begin{equation}\label{in:pSob}
\|f\|_{L^{\nu}(X)}\ls \widetilde{C_S} \left(\int_X \Gamma(f)\rd m\right)^{\frac12}.
\end{equation}
\end{lem}
\begin{proof}
The above Bishop-Gromov inequality \eqref{ineq:BG} implies the doubling property, and by Theorem 2.1 of \cite{rajala2012local}, a Poincar\'e inequality holds. These two ingredients imply the following Sobolev inequality by Theorem 9.7 of \cite{hajlasz2000sobolev}: there exists constants $\nu>2$ and $C_S>0$, depending on $K$, $N$ and $\text{diam}X$, such that for all $f\in \bV$,
\begin{equation}\label{in:Sob}
 \left(\dashint_X|f-\dashint_Xf|^{\nu}\right)^{\frac{1}{\nu}} \ls C_S\left(\dashint_X \Gamma(f)\rd m\right)^{\frac 12},
\end{equation}
where $\dashint_X \Gamma(f):=\frac{1}{m(X)}\int_X \Gamma(f) \rd m$.

Note that $\dashint_X f$ is a constant and that $f=0$ on $ E$, thus we have $\|f-\dashint_X f\|_{L^{\nu}(E)}=(m(E))^{\frac{1}{\nu}}\cdot|\dashint_X f|$ and
\[
\begin{array}{ll}
\|\dashint_X f\|_{L^{\nu}(X)}&={m(X)}^{\frac1\nu}|\dashint_X f|\\
&=\left(\frac{m(X)}{m(E)}\right)^{\frac1\nu}\|f-\dashint_X f\|_{L^{\nu}(E)}\\
&\ls \left(\frac{m(X)}{m(E)}\right)^{\frac1\nu}\|f-\dashint_X f\|_{L^{\nu}(X)}.
\end{array}
\]
Then, by Minkowski inequality, we have
\[
\begin{array}{ll}
\|f\|_{L^{\nu}(X)} &\ \ls \|f-\dashint_X f\|_{L^{\nu}(X)}+\|\dashint_X f\|_{L^{\nu}(X)}\\
&\ \ls \left[1+(\frac{m(X)}{m(E)})^{\frac1\nu}\right]\|f-\dashint_X f\|_{L^{\nu}(X)}\\
&\stackrel{\eqref{in:Sob}}{\ls} \left[1+(\frac{m(X)}{m(E)})^{\frac1\nu}\right]\cdot C_S {m(X)}^{\frac{1}{\nu}-\frac12} \left(\int_X \Gamma(f)\rd m\right)^{\frac12}.
\end{array}
\]
Let $\widetilde{C_S}= C_S  {m(X)}^{\frac{1}{\nu}-\frac12} \left[1+\frac{m(X)}{m(E)})^{\frac1\nu}\right]$, thus we have completed the proof.
\end{proof}

\section{eigenvalue estimate for $RCD^*(K,N)$-spaces}
Let $K\in \mathbb R$ and $1\ls N<\infty$, and let $(X,d,m)$ be a compact $RCD^*(K,N)$-space.
We need a version of maximum principle on $X$ as follows.
\begin{prop}\label{prop3.1}
Let $u\in \mathbb M_\infty$ and  let $\varepsilon_0>0$. If the measure $\Delta^*u$ satisfies that
the singular part $\Delta^su\gs0$ on $X$ and that the absolutely continuous part
\begin{equation}\label{eq3.1}
\Delta^{ab}u\gs C_1\cdot u-C_2\cdot \sqrt{\Gamma(u)}\quad {\rm on }\quad \{x: u(x)\gs\varepsilon_0\}
\end{equation}
holds for some positive constants $C_1$ and  $C_2 $ $($they may depend on $\varepsilon_0$$)$. Then $u\ls\varepsilon_0$ $m$-a.e. on $X.$
\end{prop}
\begin{remark}\label{rem3.2}
 If $X$ is a smooth Riemanian manifold, and if $u$ is a $C^2$-function, then the Proposition \ref{prop3.1} is a corollary of the classical maximum principle. In fact, if the assertion is false in this case, we assume that $u$ achieves its a maximum at point $p$, where $u(p)>\epsilon_0$. By using the maximum principle on $C^2$-functions, we have $$\Delta u(p)\ls0\quad {\rm and} \quad \Gamma(u)(p)=0.$$
Hence, by \eqref{eq3.1}, we have $u(p)\ls0.$ This contradicts to $u(p)> \varepsilon_0.$

In the setting of metric measure spaces, we need a new argument.
\end{remark}
\begin{proof}[Proof of Proposition \ref{prop3.1}]
 Since $u\in L^\infty(X,m)$, we have $\sup_{X} u<\infty$, where $\sup_Xu=\inf\{l: (u-l)_+=0, m{\rm-}a.e. \ {\rm in}\ X\}.$

Let us argue by contradiction. Suppose that
$\varepsilon_0<\sup_Xu.$

Take any constant $k\in[\varepsilon_0,\sup_X u)$ and set $\phi_k=(u-k)^+$. Then $\phi_k \in \bV$.
 Since singular part $\Delta^s u \gs0$, we have
\[
\begin{array}{ll}
-\int_X \Gamma(u,\phi_k)\,\rd m&\ =\int_X \widetilde{\phi_k} d\Delta^*u\\
 &\ \gs \int_X \widetilde{\phi_k} \Delta^{ab}u\,\rd m \\
 &\ = \int_{\{x: u(x)\gs k\}}  \phi_k \Delta^{ab}u\,\rd m \\
 & \overset{\eqref{eq3.1}}{\gs}  -C_2\int_{X_k} \phi_k \sqrt{\Gamma(u)}\rd m\\
&\ \gs -C_2 \left(\int_{X_k} {\phi_k}^2\right)^{\frac 12}\left(\int_{X_k} \Gamma(u)\right)^{\frac 12},
\end{array}
\]
where $X_k:=\{x: \Gamma(u)\not=0\}\cap\{x: u(x)> k\}.$

By the truncation property in \cite{sturm1994analysis} and $ \Gamma(u,\phi_k) =\Gamma(u)=\Gamma(\phi_k)$ $m$-a.e. in $X_k$, we have
\[
\int_X \Gamma(u,\phi_k)=\int_{X_k} \Gamma(u,\phi_k)=\int_{X_k} \Gamma(u)=\int_{X_k} \Gamma(\phi_k).
\]
The combination of the above two equations implies that
\begin{equation}\label{eq3.2}
\int_{X_k}\Gamma(\phi_k) \ls C_2^2 \int_{X_k} {\phi_k}^2.
\end{equation}

Now we claim that there exists a constant $ k_0\in[\varepsilon_0,\sup_X u)$  such that
\begin{equation}\label{eq3.3}
m(\{x: u(x)<k_0\})>0.
\end{equation}
Suppose that \eqref{eq3.3} fails for any $k\in[\varepsilon_0,\sup_Xu)$. That is,  $m(\{x: u(x)<k\})=0$ for any $k\in[\varepsilon_0,\sup_Xu).$ Letting $k$ tend to $\sup_Xu$, we get
$m(\{x: u(x)<\sup_Xu\})=0$. Thus $u=\sup_Xu$ $m$-a.e. in $X$. Now, we have $\Delta^*u=0$ and $\Gamma(u)=0$ $m$-a.e. in $X$. This contradicts \eqref{eq3.1} and proves the claim.

Fix such a constant $k_0\in[\varepsilon_0,\sup_Xu)$ such that \eqref{eq3.3} holds. Denote $E=\{x: u(x)<k_0\}$. For all $k\in (k_0,\sup_Xu)$, we have
$\phi_k=0$ $m$-a.e. in $E$. By applying Lemma \ref{lem2.9}, we conclude that
\begin{equation}\label{eq3.4}
\|\phi_k\|_{L^{\nu}(X)}\ls \widetilde{C_S} \left(\int_X \Gamma(\phi_k)\rd m\right)^{\frac12},\quad\ \forall k\in (k_0,{\rm sup}_Xu).
\end{equation}

We shall show that $m(X_k)>0$ for all $k\in (k_0,{\rm sup}_Xu)$. Fix any $k\in (k_0,{\rm sup}_Xu)$, the set $\{x: u(x)>k\}$ has positive measure, because $k<\sup_Xu$.
Hence, $\|\phi_k\|_{L^{\nu}(X)}>0.$ By using \eqref{eq3.4}, we get $m(\{x: \Gamma(\phi_k)\not=0\})>0.$
Note that
\begin{equation*}\Gamma(\phi_k)=
\begin{cases}
\Gamma(u) & m{\rm-}a.e. \ {\rm in}\ \{x: u(x)>k\}\\
0 & m{\rm -}a.e.\ {\rm in}\  \{x: u(x)\ls k\},
\end{cases}
\end{equation*}
we have  $\{x: \Gamma(\phi_k)\not=0\}\subset X_k$ up to a zero measure set. Thus, we get $m(X_k)\gs m(\{x: \Gamma(\phi_k)\not=0\}) >0$.

On the other hand, we have
\begin{align*}
\|\phi_k\|_{L^2(X_k)} &\ \ls \|\phi_k\|_{L^{\nu}(X_k)}\cdot (m(X_k))^{1/2-1/{\nu}} \ls \|\phi_k\|_{L^{\nu}(X )}\cdot (m(X_k))^{1/2-1/{\nu}}\\
&\ \ls \widetilde{C_S} \cdot\Big(\int_{X_k} \Gamma(\phi_k)\rd m\Big)^{1/2} (m(X_k))^{1/2-1/{\nu}}\\
&\overset{\eqref{eq3.2}}{\ls} \widetilde{C_S} \cdot C_2\cdot\|\phi_k\|_{L^2(X_k)}  (m(X_k))^{1/2-1/{\nu}},
\end{align*}
 where we have used that $\{x: \Gamma(\phi_k)\not=0\}\subset X_k$ up to a zero measure set again. Note that $m(X_k)>0$, hence  $\|\phi_k\|_{L^2(X_k)}\not=0$, for all $k\in (k_0,{\rm sup}_Xu)$,
 there is a constant $C>0$, such that $m(X_k)>C$ for all $k_0 \ls k< \sup_X u.$
Recall that $X_k=\{x: \Gamma(u)\not=0\}\cap\{x: u(x)> k\}$, by letting $k\to \sup_X u$, we have
\[
m(\{x: \Gamma(u)\not=0\}\cap \{u=\sup_X u \}) \gs C.
\]
This contradicts the fact that $\Gamma(u)=0$ a.e. in $\{u=\sup_X u \}$ (see Proposition 2.22 of \cite{cheeger1999differentiability}), and proves the proposition.
\end{proof}

Let us recall the one-dimensional model operators $L_{R,l}$ in \cite{bakry2000some}. Given $R\in\bR$ and $l>1$, the one-dimensional models $L_{R,l}$ are defined as follows: let $L=R/(l-1)$,

(1) If $R>0$, $L_{R,l}$ defined on $(-\pi/2\sqrt{L}, \pi/2\sqrt{L})$ by
\[
L_{R,l}v(x)=v''(x)-(l-1)\sqrt{L}\tan(\sqrt{L}x)v'(x);
\]

(2) If $R<0$, $L_{R,l}$ defined on $(-\infty, \infty)$ by
\[
L_{R,l}v(x)=v''(x)-(l-1)\sqrt{-L}\tan(\sqrt{-L}x)v'(x);
\]

(3) If $R=0$, $L_{R,l}$ defined on $(-\infty, \infty)$ by
\[
L_{R,l}v(x)=v''(x).
\]

Next we will apply Corollary \ref{ineq:sfi2} to eigenfunctions
 and prove the following comparison theorem on the gradient of the eigenfucntions, which is an extension of Kr\"oger's comparison result in \cite{kroger1992spectral}.
\begin{thm}\label{in:comp}
Let $(X,d,m)$ be a compact $RCD^*(K,N)$-space, and let $\lambda_1$ be  the first eigenvalue  on $X$.
Let $l\in \bR$ and $l\gs N$, and let  $f$ be an eigenfunction with respect to $\lambda_1$.  Suppose $\lambda_1>\max \left\{0,\frac{lK}{l-1}\right\}.$ Let $v$ be a Neumann eigenfunction of $L_{K,l}$ with respect to the same eigenvalue $\lambda_1$ on some interval. If $[\min f, \max f] \subset [\min v, \max v]$, then
\[
 \Gamma(f) \ls (v'\circ v^{-1})^2(f)\text{ m-a.e.}.
\]
\end{thm}

\begin{proof}
Without loss of generality, we may assume that $[\min f,\max f] \subset (\min v, \max v)$.

Denote by $T(x)$ the function such that
\[
L_{K,l}(v)=v''-Tv.
\]
As in Corollary 3 in section 4 of \cite{bakry2000some}, we can choose a smooth bounded function $h_1$ on $[\min f,\max f]$ such that
\[
h'_1<\min \{Q_1(h_1),Q_2(h_1)\},
\]
where $Q_1, Q_2$ are given by following
\[
Q_1(h_1):=-(h_1-T)\left(h_1-\frac{2l}{l-1}T+\frac{2\lambda_1 v}{v'}\right),
\]
\[
Q_2(h_1):=-h_1\left(\frac{l-2}{2(l-1)}h_1-T+\frac{\lambda_1 v}{v'}\right).
\]
We can then take a smooth function $g$ on $[\min f,\max f]$, $g\ls0$ and $g'=-\frac{h_1}{v'}\circ v^{-1}.$

According to \cite[Theorem 6.5]{ambrosio2014metric} (see also  \cite[Theorem 1.1]{jiang2014cheeger}),  we have that $f$ is Lipschitz continuous.  Notice that $\Delta f=-\lambda_1f\in \mathbb V$. Hence $f\in D_{\bV}(\Delta)\cap  Lip(X)$.

Now define a function $F$ on $X$ by
\[
\psi(f)F=\Gamma(f)-\phi(f),
\]
where $\psi(f):=e^{-g(f)}$ and $\phi(f)~:=(v'\circ v^{-1})^2(f)$.
Since $f\in  D_{\bV}(\Delta)\cap Lip(X)$, by Theorem \ref{thm:sel}, we have $\Gamma(f)\in \mathbb M_\infty$.  According to Lemma \ref{lem2.1} and Remark \ref{rem2.5}, we have  $\psi(f), \phi(f) \in D(\Delta)\cap Lip(X)$ and $F\in \mathbb M_\infty$. Moreover
\[
\Delta^*F=\frac{1}{\psi}\Delta^*\Gamma+\frac{1}{\psi}\big(-2\Gamma(\psi,F)-\Delta\psi F-\Delta \phi\big)\cdot m,
\]
where and in the sequel, we denote by $\Gamma=\Gamma(f)$ and $\phi=\phi(f),\psi=\psi(f).$
By using Theorem \ref{thm:sel} again, we have $\Delta^sF\gs 0$ on $X$ and
\[
\Delta^{ab}F=\frac{1}{\psi}\left(\Delta^{ab}\Gamma-2\Gamma(\psi,F)-\Delta\psi F-\Delta \phi\right)\quad m{\rm-}a.e.\ {\rm in}\ X.
\]

Since $l\gs N$, the $(X,d,m)$ satisfies also $RCD^*(K,l)$ condition.
Applying inequality \eqref{in:modi} to $f$ and using $\Delta f =-\lambda_1 f$, we have, for $m$-a.e. $x\in \{x: \Gamma(x)>0\}$,
\[
\Delta^{ab}\Gamma \gs -2\lambda_1\Gamma+\frac{2{\lambda_1}^2}{l}f^2+2K\Gamma+\frac{2l}{l-1}\left(\frac{\lambda_1f}{l}+\frac{\Gamma(f,\Gamma)}{2\Gamma}\right)^2.
\]

Fix arbitrarily a constant $\eps_0>0$. We want to show $F \ls \eps_0$ $m$-a.e. in $X$.

Since $F\ls e^g\cdot\Gamma\ls \Gamma$, we have  $ \{x: F(x)\gs {\eps_0}\}\subset \{x: \Gamma(x)>0\}$.
Following the argument from line 29 on page 1182 to line 10 on page 1183 of \cite{qian2013sharp}, we get:
\begin{equation}\label{ineq:main}
\Delta^{ab}F \gs \psi T_1\cdot F^2+ T_2\cdot F+T_3 \Gamma(f,F) \quad m{\rm-}a.e. \text{ on } \{x: F(x)\gs\eps_0\},
\end{equation}
where
\[
{v'}^2T_1=Q_2(h_1)-h'_1, \quad T_2=Q_1(h_1)-h'_1,
\]
and
\[
T_3=\frac{2l}{l-1}\left(-\frac{g'}{2}+\frac{1}{2\Gamma}\left(\frac{2\lambda_1f}{l}+\phi'+\phi g'\right)\right)+2g'.
\]
Note that both $T_1$ and $T_2$ are positive, $\Gamma$ is bounded on $X$ and $T_3$ is bounded on $\{x: F(x)\gs \eps_0\}$. It follows from \eqref{ineq:main} that
\begin{equation}\label{in:main}
\Delta^{ab}F \gs c_1\cdot F-c_2\cdot\sqrt{\Gamma(F)} \quad a.e. \text{ on } \{x: F(x)\gs \eps_0\}
\end{equation}
for some constant $c_2>0$ and  $c_1=\min_{s\in[\min f,\max f]} T_2(s)>0$.
By combining with $\Delta^sF\gs 0$ on $X$ and Proposition \ref{prop3.1}, we conclude that  $F \ls \eps_0$ $m$-a.e. in $X$.

At last, by the arbitrariness of $\eps_0$, we have $F \ls  0$ $m$-a.e. in $X$. This completes the proof of Theorem \ref{in:comp}.
\end{proof}

Let $v_{R,l}$ be the solution of the equation
\[
L_{R,l}v=-\lambda_1 v
\]
with initial value $v(a)=-1$ and $v'(a)=0$, where
\[
a=  \left \{
\begin{array}{ll}
-\frac{\pi}{2\sqrt{R/(l-1)}}& \qquad \text{ if }R>0,\\
0&\qquad \text{ if } R \ls 0.
\end{array}
\right.
\]
We denote
\[
b= \inf \{x>a:~v'_{R,l}(x)=0 \}
\]
and
\[
m_{R,l}=v_{R,l}(b)
\]
Note that $v_{R,l}$ is non-decreasing on $[a,b]$.

Next we show the following comparison theorem on the maximum of eigenfunctions.
\begin{thm}\label{thm3.4}
Let $(X,d,m)$ be a compact $RCD^*(K,N)$-space, and let $f$ be an eigenfunction with respect to the first eigenvalue $\lambda_1$ on $X$.
Suppose $\min f=-1, \max f \ls 1$. Then we have
\[
\max f \gs m_{K,N}.
\]
\end{thm}

\begin{proof}
We argue by contradiction. Suppose $\max f<m_{K,N}$.
Since $m_{K,l}$ is continuous on $l$, we can find some real number $l>N$ such that
\[
\max f \ls m_{K,l}\text{ and }\lambda_1>\max\{0,\frac{lK}{l-1}\}.
\]

Then following the proof of Proposition 5 in \cite{bakry2000some}, we obtain that the ratio
\[
R(s)=-\frac{\int_X f 1_{\{f\ls v(s)\}}\rd m} {\rho(s)v'(s)}
\]
is increasing on $[a,v^{-1}(0)]$ and decreasing on $[v^{-1}(0),b]$, let $L=K/(l-1)$, the function $\rho$ is
\[
\rho(s)~:=\left \{
\begin{array}{ll}
\cos^{l-1}(\sqrt{L}s)&  \text{ if } L>0\\
s^{l-1}&  \text{ if } L=0\\
\sinh^{l-1}(\sqrt{-L}s)&  \text{ if } L<0.
\end{array}
\right.
\]
It follows that for any $s\in[a,v^{-1}(-1/2)]$, since $v(s)\ls -\frac 12$, we have
\begin{equation} \label{ineq:3.31}
m(\{f\ls v(s)\})\ls -2\int_X f 1_{\{f\ls v(s)\}} \rd m \ls 2C \rho(s) v'(s),
\end{equation}
where $C=R(v^{-1}(0))$.

Take $p \in X$ with $f(p)=-1$. By
\[
f-f(p) \gs 0,\quad \text{ and } \Delta(f-f(p))=-\lambda_1 f \ls \lambda_1
\]
The mean value inequality \eqref{in:mv}, implies that
\[
\dashint_{B_r(p)} (f-f(p)) \rd m \ls C \lambda_1 r^2
\]
for all $r>0$ such that $B_r(x)\subset X$.
Denote $C_1= C \lambda_1$.
Let $A(r)=\{f-f(p)>2C_1 r^2\} \cap B_p(r)$. Then
\[
\frac{m (A(r))}{m (B_p(r))} \ls \frac{\int_{B_r(p)}(f-f(p)) \rd m}{2C_1r^2 m(B_r(p))} \ls \frac12.
\]
Hence
\[
\begin{array}{ll}
\frac12 m(B_r(p)) &\ls m(B_r(p)\backslash A(r)) \\
&\ls m(\{f-f(p) \ls 2C_1 r^2\})\\
&= m(\{f \ls -1+ 2C_1 r^2\}).
\end{array}
\]
By using \eqref{ineq:3.31} and following the argument from line 1 on Page 1186 to line 3 on page 1187 of \cite{qian2013sharp}, one can get that there exists a constant $C_2>0$ such that
\[
m(B_p(r)) \ls C_2 r^l
\]
for all sufficiently small $r>0$.

Fix $r_0>0$. By Bishop-Gromov inequality \eqref{ineq:BG}, we have

\[
m(B_p(r)) \gs \frac{m(B_p(r_0))}{\int_0^{r_0} \fs_{\frac{K}{N-1}}(t)^{N-1} \rd t}\int_0^r \fs_{\frac{K}{N-1}}(t)^{N-1} \rd t \gs C_3 r^N
\]
for any $0<r<r_0$.
The combination of the above two inequalies implies that $C_2 r^{l-N} \gs C_3$ holds for any sufficiently small $r$. Hence, we have $l \ls N$, which contradicts to the assumption $l >N$. Therefore, the Proof
of Theorem \ref{thm3.4} is finished.
\end{proof}

Now we are in the position to prove the main result---Theorem \ref{thm:main}.
\begin{proof}[Proof of Theorem \ref{thm:main}]
%Without loss of generality, we may assume $K\in \{-(N-1),0,N-1\}$.
 Let $\lambda_1$ and $f$ denote respectively the first non-zero eigenvalue and a corresponding eigenfunction with $\min f=-1$ and $\max f \ls 1$.
By Theorem 4.22 of \cite{erbar2013equivalence}, we have $\lambda_1 \gs NK/(N-1)$ if $K>0$ and $N>1$. Now fix any $R<K$, we have
 \[
\lambda_1>\max \left \{\frac{NR}{N-1},0\right \}.
\]

Then we may use the results of Section 6 and Section 3 of \cite{bakry2000some}, we can find an interval $[a,b]$ such that the one-dimensional model operator $L_{R,N}$ has the first Neumann eigenvalue $\lambda_1$ and a corresponding eigenfunction $v$ with $v(a)=\min v=-1$ and $v(b)=\max v=\max f.$
By Theorem 13 in Section 7 of \cite{bakry2000some}, we have
\begin{equation}\label{in:last}
\lambda_1 \gs \hla(R,N,b-a),
\end{equation}
where $\hla(R,N,b-a)$ is the first non-zero Neumann eigenvalue of $L_{R,N}$ on the symmetric interval $(-\frac{b-a}{2}, \frac{b-a}{2})$.
Note that $f$ is continuous, we take $x$ and $y$ are two points on $X$ such that $f(x)=-1$ and $f(y)=\max f $.
 Let $g=v^{-1}\circ f$,
  then $g(x)=a$, $g(y)=b$ and, by Theorem \ref{in:comp}, $\Gamma(g) \ls1$ $m$-a.e. in $X$. Hence, we have
\[
b-a=g(y)-g(x) \ls d(x,y) \ls \max_{z_1,z_2 \in X}d(z_1,z_2):=d,
\]
where $d$ is the diameter of $X$.
Together with \eqref{in:last} and the fact that the function $\hla(R,N,s)$ decreases with s, we conclude
\[
\lambda_1 \gs \hla(R,N,d).
\]
By the arbitrariness of $R$, we finally prove the theorem.
\end{proof}

\noindent\textbf{Acknowledgements.} The second author is partially supported by NSFC 11201492.


\begin{thebibliography}{99}
\bibitem{ambrosio2015bakry} L. Ambrosio, N. Gigli, G. Savar\'e, \emph{Bakry--Emery curvature-dimension condition and Riemannian Ricci curvature bounds}, Ann. Probab., 43 (2015), 339--404.

\bibitem{ambrosio2014calculus} L. Ambrosio, N. Gigli, G. Savar\'e, \emph{Calculus and heat flow in metric measure spaces and applications to spaces with Ricci bounds from below}, Invent. Math., 195(2) (2014), 289--391.

\bibitem{ambrosio2014metric} L. Ambrosio, N. Gigli, G. Savar\'e, \emph{Metric measure spaces with Riemannian Ricci curvature bounded from below}, Duke. Math. J., 163(7) (2014), 1405--1490.

\bibitem{ambrosio2013bakry} L. Ambrosio, A. Mondino, G. Savar\'e, \emph{On the Bakry-\'Emery condition, the gradient estimates and the local-to global property of $RCD^*(K,N)$ metric measure spaces}, J. Geom. Anal., 2013, 1--33.

\bibitem{andrews2013sharp} B. Andrews, J. Clutterbuck, \emph{Sharp modulus of continuity for parabolic equations on manifolds and lower bounds for the first eigenvalue}, Anal. PDE., 6(5) (2013), 1013--1024

\bibitem{bakry2000some} D. Bakry, Z. Qian, \emph{Some new results on eigenvectors via diemension, diameter, and Ricci curvature}, Adv. in Math., 155 (2000), 98--153.

\bibitem{besson2004isoperimetric} G. Besson, \emph{From isoperimetric inequalities to heat kernels via symmetrisation},   Surveys in Differential Geometry IX: Eigenvalues of Laplacians and other geometric operators, Edited by A. Grigor'yan and S. T. Yau, (2004), pp. 27--51, International Press, Somerville, MA.


\bibitem{cheeger1970lower} J. Cheeger, \emph{A lower bound for the smallest eigenvalue of the Laplacian, Problems in analysis,} Symposium in Honor of S. Bochner, pp. 195--199, Princeton Univ. Press, Princeton, NJ, 1970.

\bibitem{cheeger1999differentiability} J. Cheeger, \emph{Differentiability of Lipschitz functions on metric measure spaces}, Geom. Funct. Anal. 9,  (1999), 428--517.

\bibitem{chen1993application} M. F. Chen, F. Y. Wang, \emph{Application of coupling method to the first eigenvalue on manifold}, Sci. Sin. (A), 37(1994)
1--14.

\bibitem{chen1997general} M. F. Chen, F. Y.  Wang, \emph{General formula for lower bound of the first eigenvalue on Riemannian manifolds}, Sci.
Sin. (A), 40(1997) 384--394.

\bibitem{colding2011sharp} T. Coding, A. Naber, \emph{Sharp H\"older continuity of tangent cones for spaces with a lower Ricci curvature bound and applications}, Ann. Math., 176 (2012), 1173--1229 .

\bibitem{davies} E. B. Davies, \emph{Heat kernels and spectral theory}, Cambridge Tracts in Mathematics, vol. 92, Cambridge University Press, Cambridge, (1990).

\bibitem{erbar2013equivalence} M. Erbar, K. Kuwada, K. Sturm,
\emph{On the equivalence of the entropic curvature-dimension condition and Bochner¡¯s inequality on metric measure spaces}, Invent. Math., (2014, to appear).

\bibitem{garofalo2014li} N. Garofalo, A .Mondino, \emph{Li-Yau and Harnack type inequalities in metric measure spaces}, Nonlinear Analysis: Theory, Methods and Applications, 95(2014), 721--734.

%\bibitem{gigli2012differential} N. Gigli, \emph{On the differential structure of metric measure spaces and applications}, available at http://arxiv.org/abs/1205.6622.

\bibitem{gigli2013optimal} N. Gigli, T. Rajala, K. Sturm, \emph{Optimal maps and exponentiation on finite dimensional spaces with Ricci curvature bounded from below}, available at http://arxiv.org/abs/1305.4849.

\bibitem{hajlasz2000sobolev} P. Hajlasz, P. Koskela, \emph{Sobolev met Poincar\'e}, Memiors of American Mathmetical Soc. 688(2000), 1--101.

\bibitem{jiang2014cheeger} R. Jiang, \emph{Cheeger-harmonic functions in metric measure spaces revisited}, J. Funct. Anal., 266(3)(2014), 1373--1394.

\bibitem{ketterer2014cones} C. Ketterer, \emph{Cones over metric measure spaces and the maximal diameter theorem}, J. Math. Pure. Appl., (to appear), 2014.

\bibitem{ketterer2014obata} C. Ketterer, \emph{Obata's rigidity theorem for metric measure spaces}, available at http://
arxiv.org/abs/1410.5210.

%\bibitem{kinnunen2002nonlinear} J. Kinnunen, O. Martio, \emph{Nonlinear potential theory on metric spaces}, Illinois J. Math., 46(3) (2002), 857-883.

\bibitem{kroger1992spectral} P. Kr\"oger, \emph{On the spectral gap for compact manifolds}, J. Diff. Geom. 36 (1992), 315--330.


\bibitem{ledoux2004spectral} M. Ledoux, \emph{Spectral gap, logarithmic Sobolev constant, and geometric bounds},   Surveys in Differential Geometry IX: Eigenvalues of Laplacians and other geometric operators, Edited by A. Grigor'yan and S. T. Yau, (2004), pp. 219--240, International Press, Somerville, MA.

\bibitem{li1980estimates} P. Li, S. T. Yau, \emph{Estimates of eigenvalues of a compact Riemannian manifold}, In: AMS Symposium on Geometry of
the Laplace Operator, XXXVI. Hawaii, pp. 205--240 (1979).

\bibitem{lichnerowicz1958geometrie} A. Lichnerowicz, \emph{G\^{e}ometries des Groupes des Transformations,}  Paris, Dunod, 1958.

\bibitem{lott2009ricci} J. Lott, C. Villani,  \emph{Ricci curvature for metric-measure spaces via
optimal transport}, Ann. Math. 169 (2009), 903--991.

\bibitem{lott2007weak} J. Lott, C. Villani,  \emph{Weak curvature bounds and functional inequalities}, J. Funct. Anal. 245(1) (2007), 311--333.

\bibitem{mondino2014structure} A. Mondino, A. Naber, \emph{Structure theory of metric measure spaces with lower Ricci curvature bounds I},

\bibitem{petruninharmonic}  A. Petrunin, \emph{Harmonic functions on Alexandrov space and its
applications,}
ERA Amer. Math. Soc., 9 (2003), 135--141.

\bibitem{qian2013sharp} Z. Qian, H. C. Zhang, X. P. Zhu, \emph{Sharp spectral gap and Li-Yau's estimate on Alexandrov spaces}, Math. Z., 273(3-4) (2013), 1175--1195.

\bibitem{rajala2012local} T. Rajala, \emph{Local Poincar\'e inequalities from stable curvature conditions on metric spaces}, Calc. Var. Partial. Dif., 44(3-4) (2012), 477--494.

\bibitem{savare2013self} G. Savar\'e, \emph{Self-improvement of the Bakry-\'Emery condition and Wasserstein contraction of the heat flow in $RCD (K,\infty)$ metric
 measure spaces}, Disc. Cont. Dyn. Syst. A, 34 (2014), 1641--1661.

\bibitem{schoen1994lectures} R. Schoen, S. T. Yau, \emph{Lectures on Differential Geometry},
International Press, Boston, 1994.

\bibitem{shio99} T. Shioya, \emph{Eigenvalues and suspension structure of compact Riemannian orbifolds with positive Ricci
curvature}, Manuscripta Math 99(4), (1999), 509--516.


\bibitem{sturm1994analysis} K. Sturm, \emph{Analysis on local Dirichlet spaces. I. Recurrence, conservativeness and $L^p$-Liouville properties.}, J. Reine Angew. Math., 456 (1994), 173--196.

%\bibitem{sturm1996analysis} K. Sturm,  \emph{Analysis on local Dirichlet spaces. III. The parabolic Harnack inequality}, J. Math. Pure. Appl., 75(3) (1996), 273--297 .

\bibitem{sturm2006geometry} K. Sturm, \emph{On the geometry of metric measure spaces}, Acta. Math., 196(1) (2006), 65--131.

\bibitem{sturm2006geometry2} K. Sturm, \emph{On the geometry of metric measure spaces II}, Acta. Math., 196(1) (2006), 133--177.

\bibitem{wang2013sharp} G. Wang, C. Xia, \emph{A sharp lower bound for the first eigenvalue on Finsler manifolds}, Annales de l'Institut Henri Poincare (C) Non Linear Analysis, 30(6) (2013), 983--996.


\bibitem{huichun2010new} H. C. Zhang, X. P.  Zhu, \emph{Ricci curvature on Alexandrov spaces and rigidity theorems}, Comm. Anal.
Geom. 18(3), (2010) 503--554.

\bibitem{ZhongYang1984estimate} J. Q. Zhong, H. C. Yang, \emph{On the estimate of the first eigenvalue of a compact Riemannian manifold,} Sci. Sinica Ser. A 27(12) (1984), 1265--1273.

\end{thebibliography}
\end{document}